\documentclass[12pt]{article}

\usepackage{amsmath,amsthm,amssymb,amscd,titlesec}
\usepackage{enumerate}
\usepackage[all,cmtip]{xy}
\usepackage[backref]{hyperref}

\titleformat{\paragraph}
{\normalfont\normalsize\bfseries}{\theparagraph}{1em}{}
\titlespacing*{\paragraph}
{0pt}{3.25ex plus 1ex minus .2ex}{1.5ex plus .2ex}

\usepackage[margin=1in]{geometry}

\AtEndDocument{\bigskip{\footnotesize%
  \textsc{School of Mathematical Sciences, Tel-Aviv University} \par  
  \textit{E-mail address}: \texttt{evgenym@post.tau.ac.il} \par
  \addvspace{\medskipamount}
  \textsc{School of Mathematical Sciences, Tel-Aviv University} \par  
  \textit{E-mail address}: \texttt{yomdina@post.tau.ac.il} \par
}}

\newtheorem{theorem}{Theorem}[section]
\newtheorem{lemma}[theorem]{Lemma}
\newtheorem*{theorem*}{Theorem}

\theoremstyle{definition}
\newtheorem{definition}[theorem]{Definition}

\newtheorem{fact}[theorem]{Fact}

\newtheorem{claim}[theorem]{Claim}
\newtheorem{corollary}[theorem]{Corollary}

\theoremstyle{remark}
\newtheorem*{remark*}{Remark}
\newtheorem{remark}[theorem]{Remark}

\numberwithin{equation}{section}

\newcommand {\Z} {\mathbb{Z}}

\newcommand {\bbS} {\mathbb{S}}

\newcommand {\niceC} {\mathcal{C}}

\newcommand {\niceF} {\mathcal{F}}
\newcommand {\niceG} {\mathcal{G}}

\newcommand {\niceO} {\mathcal{O}}

\newcommand {\niceS} {\mathcal{S}}
\newcommand {\niceT} {\mathcal{T}}
\newcommand {\niceU} {\mathcal{U}}
\newcommand {\niceV} {\mathcal{V}}

\newcommand {\lien} {\mathfrak{n}}
\newcommand {\liem} {\mathfrak{m}}

\newcommand {\liep} {\mathfrak{p}}
\newcommand {\lieq} {\mathfrak{p}}

\begin{document}

\title{Reciprocity laws and $K$-theory}
\author{Evgeny Musicantov, Alexander Yom Din}

\date{}
\maketitle

\begin{abstract}

We associate to a full flag $\niceF$ in an $n$-dimensional variety $X$ over a field $k$, a "symbol map" $\mu_{\niceF}:K(F_X) \to \Sigma^n K(k)$. Here, $F_X$ is the field of rational functions on $X$, and $K(\cdot)$ is the $K$-theory spectrum. We prove a "reciprocity law" for these symbols: Given a partial flag, the sum of all symbols of full flags refining it is $0$. Examining this result on the level of $K$-groups, we derive the following known reciprocity laws: the degree of a principal divisor is zero, the Weil reciprocity law, the residue theorem, the Contou-Carr\`{e}re reciprocity law (when $X$ is a smooth complete curve) as well as the Parshin reciprocity law and the higher residue reciprocity law (when $X$ is higher-dimensional).

\end{abstract}

\pagebreak
\tableofcontents
\pagebreak


\section{Introduction}

\subsection{Overview}

Several statements in number theory and algebraic geometry are "reciprocity laws". Let us consider, as an example, the Weil reciprocity law. Let $X$ be a complete smooth curve over an algebraically closed field $k$, and let us fix $f,g \in F_X^{\times}$, two non-zero rational functions on $X$. Given a point $p \in X$, one defines the tame symbol: $$ ( f , g )_p := (-1)^{v_p (f) \cdot v_p (g)} \frac {f^{v_p (g)}}{g^{v_p (f)}} (p).$$ Here, $v_p$ is the valuation at $p$ (order of zero). The Weil reciprocity law states that \linebreak $(f,g)_p = 1$ for all but finitely many $p \in X$, and that $\prod_{p \in X} (f,g)_p = 1$. More generally, one can describe the pattern as follows: There is a global object, exhausted by local pieces. One then associates an invariant to each local piece, as well as to the global object itself. The desired claim is then two-fold:

\begin {enumerate}[(i)]
\item ("Global is trivial") The global invariant is trivial.
\item ("Local to Global") The product of the local invariants equals the global invariant (usually this is an infinite product, and one should figure out how to make sense of it).
\end {enumerate}

In the above example, the global object is the curve $X$, which is exhausted by the local pieces -- the points of the curve. The invariant associated to a local piece is the tame symbol, while the global invariant is quite implicit.

Let us recall that the Weil reciprocity law admits a higher-dimensional analog, known as the Parshin reciprocity law (\cite[Appendix A]{Sop}; see also \ref{par:Parshin law}).

In this text we define symbol maps and prove a reciprocity law, using the machinery of algebraic $K$-theory. We then see how various reciprocity laws, such as the Parshin reciprocity law (generalizing the Weil reciprocity law), the higher residue reciprocity law (generalizing the residue theorem), the Contou-Carr\`{e}re reciprocity law, all follow from this one reciprocity law.

Let us describe our setup in more detail. Fix an $n$-dimensional irreducible variety $X$ over a field $k$ \footnote{These assumptions on $X$ and $k$ are made here merely to simplify matters, and will be relaxed below.}. By a \textit{full flag} $\niceF$ in $X$ we mean a chain of closed irreducible subvarieties $X = X^0 \supset X^1 \supset \ldots \supset X^n$, where the codimension of $X^i$ in $X$ is $i$. Given a full flag $\niceF$, we shall define a morphism of spectra (we call it a symbol map) $$ \mu_{\niceF} : K(F_X) \to \Sigma^n K (k).$$ Here $F_X$ is the field of rational functions on $X$, $K(\cdot)$ denotes the $K$-theory spectrum, and $\Sigma^n$ denotes $n$-fold suspension. By a \textit{partial flag} $\niceG$ in $X$ we mean a full flag with an element in some single codimension $0 < d \leq n$ omitted. Then, given a partial flag $\niceG$, we may consider the set $fl(\niceG)$ of full flags which refine it. The main result of this text, theorem \ref{MainTheorem}, then states:

\begin {theorem*}

Let $X$ be an $n$-dimensional irreducible variety over a field $k$. Let $$\niceG: X^0 \supset \cdots \supset X^{d-1} \supset X^{d+1} \supset \cdots \supset X^n$$ be a partial flag in $X$, with element in codimension $0<d\leq n$ omitted. In the case $d=n$, assume additionally that the curve $X^{n-1}$ is proper over $k$. Then, $$ \sum_{\niceF \in fl(\niceG)} \mu_{\niceF} = 0.$$

\end {theorem*}

\begin {remark*}
The sum figuring in the theorem is infinite, however in appendix \ref{App:AppendixA} we will make sense of it (inspired by \cite{Cl}).
\end {remark*}

In fact, it is more "correct" to additionally define a symbol map $$ \mu_{\niceG} : K(F_X) \to \Sigma^n K (k),$$ associated to a partial flag $\niceG$. The theorem then divides into two parts; that $\mu_{\niceG}$ equals zero, and that the sum of all the morphisms $\mu_{\niceF}$ for $\niceF \in fl(\niceG)$ equals $\mu_{\niceG}$.

Notice how this setup instantiates the general pattern above. A fixed partial flag is the global object, exhausted by the local pieces which are the full flags refining the given partial flag. The symbol map is the associated invariant.

In order to derive the concrete reciprocity laws promised above from this abstract one, one considers its effect on $K$-groups.

Let us note that, in principle, the symbol map between spectra appears to contain more information than its "shadows" on $K$-groups. However, in this text we have only re-obtained known reciprocity laws from it.

Let us record here that relevant and independent work has been done in \cite{BGW1} and \cite{BGW2}.

There are several further directions which one may consider. For example, one may consider the "curve" $Spec(\Z)$. Could our setup be altered, so that the Hilbert reciprocity law would fit in? For that to succeed, at least three phenomena should be addressed; The prime at infinity, ramification at the prime $2$ and the sphere spectrum, which underlies all primes. A relevant treatment of the case of $Spec(\Z)$ is in \cite{Cl}.

\subsection{Relation to $n$-Tate vector spaces}

There is a strong relation of our mechanism with the theory of $n$-Tate vector spaces. In fact, $n$-Tate vector spaces could be seen as the actual "geometric" objects that the target of our symbol map $\mu_{\niceF}$ classifies, so that, in a sense, our approach "decategorifies" the actual picture.

The technical result underlying such a connection is the following one. Let $\niceC$ be an exact category, and $Tate(\niceC)$ the exact category of "pro-ind" objects in $\niceC$, introduced by Beilinson (\cite{B2}).

\begin {theorem*}[\cite{Sa}]
$$ K(Tate(\niceC)) \approx \Sigma K(\niceC).$$
\end {theorem*}

Thus, we can say that the Tate consutrction acts as a delooping, when one passes to $K$-theory spectra.

In this paper we associate to a full flag $\niceF$ in an $n$-dimensional variety $X$ a symbol map $$\mu_{\niceF} : K(F_X) \to \Sigma^n K(k).$$ Taking the above theorem into account, one might interpret it as a map $$\mu_{\niceF} : K(F_X) \to K(Tate^n(k)),$$ where $Tate^n(k)$ is the $n$-fold application of the $Tate(\cdot)$ construction to the exact category $Vect(k)$ of finite-dimensional vector spaces over $k$. At this point, one might wonder whether this map comes from a functor $Vect(F_X) \to Tate^n(k)$. Indeed, such a functor can be constructed, and is essentially the adelic construction of \cite{B1}.

We will address and develop the above interesting ideas elsewhere.

Once again, we point out relevant work that has been done in \cite{BGW1} and \cite{BGW2}.

\subsection{Organization}

This text is organized as follows. Section \ref{Sec:Statements} contains the formulation of the abstract reciprocity law (subsection \ref{SubSec:AbstractRecLaw}) and the formulations of concrete reciprocity laws (subsection \ref{SubSect:ConcreteRecLaw}) which are obtained from the abstract reciprocity law by considering its effect on specific $K$-groups. Section \ref{Sec:ConstructionAndProof} contains the construction of the abstract symbol map (subsection \ref{SubSec:Construction}) and the proof of the abstract reciprocity law (subsection \ref{SubSec:ProofRecLaw}). Section \ref{Sec:Calc} deals with calculation of the symbol map on specific $K$-groups. In appendix \ref{App:AppendixA}, we describe how to make sense of an infinite sum of morphisms of spectra. In appendix \ref{App:AppendixB}, we state some lemmas about $K$-theory, which are used in calculations.

\subsection{Notations}

\addtocontents{toc}{\protect\setcounter{tocdepth}{1}}
\subsubsection{}
\addtocontents{toc}{\protect\setcounter{tocdepth}{3}}

We use \cite{TT} as a reference for $K$-theory of schemes. Given a Noetherian scheme $X$, $K(X)$ will denote the $K$-theory spectrum of $X$. Given a closed subset $Z \subset X$, $K(X \ on \ Z)$ will denote the $K$-theory spectrum of $X$ with supports in $Z$. By abuse of notation, given a commutative ring $A$ and an ideal $I \subset A$, we will also write $K(A)=K(X)$ and $K(A \ on \ I) = K(X \ on \ Z)$ where $X = Spec(A)$ and $Z \subset X$ is the closed subset associated to the ideal $I$.

\addtocontents{toc}{\protect\setcounter{tocdepth}{1}}
\subsubsection{}
\addtocontents{toc}{\protect\setcounter{tocdepth}{3}}

We use some notation for the scheme $X$ in this text:

\begin {itemize}

\item $n=dim(X)$ will denote the Krull dimension of $X$.

\item $|X|$ will denote the underlying topological space of $X$. The usual partial order on $|X|$ (that of "containment in the closure of") will be denoted by $\leq$. $|X|^i$ will denote the subset of $|X|$ consisting of points of codimension $i$.

\item $\gamma$ will denote the generic point of $|X|$ ($X$ will be assumed to be irreducible) -- i.e., the only point in $|X|^0$, and $F = F_X = \niceO_{X, \gamma}$ will denote the local ring at that point.

\item For $p \in |X|$, we will write $X_p := Spec(\niceO_{X,p})$. There is a canonical map $X_p \to X$. As usual, we will write $k(p)$ for the residue field of $\niceO_{X,p}$.

\item If $X$ is affine and $\lieq$ is a prime ideal in $\niceO(X)$, then $p_{\lieq} \in |X|$ will denote the corresponding point.

\end {itemize}

\subsection{Acknowledgments}

We thank our advisor Joseph Bernstein. We thank Eitan Sayag for showing interest and moral support. We thank Lior Bary-Soroker for valuable comments and suggestions. We thank Amnon Besser and Amnon Yekutieli for taking interest. We thank Roy Ben-Abraham, Efrat Bank, Adam Gal, Lena Gal for helpful comments.


\section{Statements} \label{Sec:Statements}

\subsection{The abstract reciprocity law} \label{AssumptionsOnXandB}
\label{SubSec:AbstractRecLaw}

Let $X \to B$ be a morphism of schemes. We make the following assumptions:

\begin{enumerate}

\item $B$ is Noetherian, $0$-dimensional (i.e., a finite disjoint union of $Spec$'s of local Artinian rings).

\item $X$ is Noetherian, of finite Krull dimension and irreducible.

\item\label{finitenessassumption} For every $p \in |X|^n$ (recall $n:=dim(X)$), the composition $Spec(k(p)) \to X \to B$ is a finite morphism.

\end{enumerate}

We give two examples of morphisms that satisfy the above assumptions.

\begin{enumerate}
\item $B=Spec(k)$ where $k$ is a field, and $X\to B$ is an irreducible scheme of finite type over $B$.
\item $B=Spec(k)$ where $k$ is a field, and $X=Spec(A)$ where $(A,\liem)$ is a Noetherian local integral $k$-algebra, such that $A/\liem$ is finite over $k$. $X \to B$ is the corresponding structure map.
\end{enumerate}

A convenient technical notion will be that of a collection $C$, by which we mean a family $C = (C^i)_{0 \leq i \leq n}$, where $C^i \subset |X|^i$. We only consider collections which satisfy $C^0 = \{ \gamma \}$.

Given a collection $C$ we will construct, in subsection \ref{SubSec:Construction}, a map of spectra ("symbol map"):

$$\mu_{C} : K(F) \to \Sigma^n K(B).$$

We will only consider and use collections attached to full and partial flags (to be now defined), for which we will state a reciprocity law. First, let $$\niceF: x_n < x_{n-1} < \ldots < x_0 = \gamma$$ be a full flag  of points in $|X|$ (thus, $codim (x_i) = i$). We define a collection $C(\niceF)$, by setting $C(\niceF)^i = \{x_i\}$. Second, let $$\niceG: x_n < x_{n-1} < \ldots < x_{d+1} < x_{d-1} < \ldots < x_0 = \gamma$$ be a partial flag , with the level $0 < d \leq n$ omitted. Here, we require $codim (x_i) = i$. We define a collection $C(\niceG)$, by setting $C(\niceG)^i = \{ x_i \}$ for $i \neq d$, and \linebreak$C(\niceG)^d = \{ p \in |X|^d \quad | \quad x_{d+1} < p < x_{d-1} \}$. Note that we have the obvious notion of a full flag refining a partial flag (if $C(\niceF) \subset C(\niceG)$), which we will denote by $\niceF > \niceG$. We will sometimes write $\mu_{\niceF}$ instead of $\mu_{C(\niceF)}$.

We prove the following "reciprocity" laws (for the meaning of the infinite sum in this statement, consult Appendix A).

\begin {theorem} \label{MainTheorem}

Let $\niceG$ be a partial flag with level $d$ omitted, where $0 < d \leq n$.

\begin {enumerate}

\item ("Global is trivial") $$\mu_{C(\niceG)} = 0,$$ where in case $d = n$ we should assume that $\overline {x_{n-1}}$ is proper over $B$.

\item ("Local to Global") $$\mu_{C(\niceG)} = \sum_{\niceF > \niceG} \mu_{C(\niceF)}.$$

\end {enumerate}

\end {theorem}

\subsection {Concrete reciprocity laws} \label{SubSect:ConcreteRecLaw}

In the following, we give examples of concrete reciprocity laws which one obtains by considering the effect of the abstract reciprocity law on various homotopy groups of the involved spectra.

\subsubsection {The case $dim(X)=1$}

Let $k$ be a field, $B=Spec(k)$ and $X \to B$ a regular, connected, proper curve over $B$. We obtain, for every closed point $p \in |X|^1$, a map $\mu_p : K(F) \to \Sigma K(B)$. Here $\mu_p = \mu_{C(\niceF)}$, where $\niceF: p < \gamma$. Applying the functor $\pi_i$, one has maps $\mu^i_p : K_i(F) \to K_{i-1}(k)$.

\paragraph {The degree law}

We have the map $\mu^1_{p} : F^{\times} \cong K_1 (F) \to K_0 (k) \cong \Z$.

\begin {claim}
The integer $\mu^1_{p} (f)$ is equal to the valuation $v_p(f)$ of $f$ at the point $p$, multiplied by $[ k(p) : k ]$.
\end {claim}

Applying the abstract reciprocity law, we recover the theorem about sum of degrees (\cite[II.3 , Prop. 1]{Ser}):

\begin {corollary}

For $f \in F^{\times}$, $$ \sum_{p \in |X|^1} [ k(p) : k ] \cdot v_p (f) = 0. $$

\end {corollary}

\paragraph {The Weil reciprocity law}

Precomposing the map $\mu^2_{p} : K_2 (F) \to K_1 (k)$ with the product in $K$-theory $K_1 (F) \wedge K_1 (F) \to K_2(F)$, we get a bilinear anti-symmetric form $\mu^2_{p} : F^{\times} \wedge F^{\times} \to k^{\times}$ (we also call it $\mu^2_p$, by abuse of notation).

\begin {claim}

$$\mu^2_{p} (f\wedge g) = N_{k(p)/k} \left( (-1)^{v_p (f) \cdot v_p (g)} \frac {f^{v_p (g)}}{g^{v_p (f)}} (p) \right).$$

\end {claim}

Applying the abstract reciprocity law, we recover the Weil reciprocity law (\cite[III.4]{Ser}):

\begin {corollary}

For $f,g \in F^{\times}$, $$ \prod_{p \in |X|^1} N_{k(p)/k} \left( (-1)^{v_p (f) \cdot v_p (g)} \frac {f^{v_p (g)}}{g^{v_p (f)}} (p) \right) = 1. $$

\end {corollary}

\paragraph {The residue law} \label{par:residue law}

Suppose now that $k$ is algebraically closed. Set $k_{\epsilon} := k[\epsilon_1 , \epsilon_2] / (\epsilon_1^2 , \epsilon_2^2)$, $B_{\epsilon} = Spec(k_{\epsilon})$ and $X_{\epsilon} = k_{\epsilon} \otimes_{k} X$. Then, the local ring at the generic point of $X_{\epsilon}$ is just $F_{\epsilon} = k_{\epsilon} \otimes_k F$. By applying our construction to the morphism $X_{\epsilon} \to B_{\epsilon}$ we get a map $K(F_{\epsilon}) \to \Sigma K(k_{\epsilon})$ for every closed point $p \in |X_{\epsilon}|^1 = |X|^1$. Applying the functor $\pi_2$ and using the product in $K$-theory as before, one gets a pairing $r_p: F_{\epsilon}^{\times} \wedge F_{\epsilon}^{\times} \to k_{\epsilon}^{\times}$.

\begin {claim}

$$ r_p((1-\epsilon_1 f) \wedge (1-\epsilon_2 g)) =  1-\epsilon_1 \epsilon_2 Res_p(f\cdot dg).$$

Here, $Res_p$ is the usual residue (\cite[II.7]{Ser}).

\end {claim}

Applying the abstract reciprocity law, we recover the residue theorem (\cite[II.7 , Prop. 6]{Ser}):

\begin {corollary}

For $f,g \in F$, $$ \sum_{p \in |X|^1} Res_p(f \cdot dg) = 0. $$

\end {corollary}

\paragraph {The Contou-Carr\`{e}re reciprocity law}

Let, more generally, $k$ be a local Artinian ring. Set $B=Spec(k)$ and $X=Spec(k[[t]])$. Applying the functor $\pi_2$ to the symbol map $K(k((t))) \to \Sigma K(k)$, one gets a pairing \linebreak $k((t))^{\times} \wedge k((t))^{\times} \to k^{\times}$. Although we don't spell out the details in this text, one can check that it is the Contou-Carr\`{e}re symbol (\cite{CC}). Then the abstract reciprocity law implies the Contou-Carr\`{e}re reciprocity law.

\subsubsection{The case $dim(X) > 1$}

Let $k$ be a field, let $B = Spec(k)$ and let $X \to B$ be an irreducible scheme of finite type over $B$ (recall $n := dim(X)$). For every full flag $\niceF$ one has a map $\mu_{\niceF} : K(F) \to \Sigma^n K(B)$. Applying the functor $\pi_i$, one gets maps $\mu^i_{\niceF} : K_i (F) \to K_{i-n} (k)$.

\paragraph {The Parshin reciprocity law} \label{par:Parshin law}

Let us assume that the flag $\niceF = x_n < x_{n-1} < \ldots < x_0 = \gamma$ is regular in the following sense: Considering $X^i := \overline{x_i}$ as an integral closed subscheme of $X$, we demand $\niceO_{X^{i-1} , x_{i}}$ to be regular (here, $1 \leq i \leq n$).

Precomposing the map $\mu^{n+1}_{\niceF} : K_{n+1} (F) \to K_1 (k)$ with the product in $K$-theory \linebreak  $\bigwedge^{n+1} K_1 (F) \to K_{n+1}(F)$, one has a multilinear anti-symmetric form $\mu^{n+1}_{\niceF} : \bigwedge^{n+1} F^{\times} \to k^{\times}$ (we also denote it $\mu^{n+1}_{\niceF}$, by abuse of notation).

In order to write an explicit formula for the Parshin symbol, we introduce the following (see \cite[Appendix A]{Sop}). For every $1 \leq i \leq n$, let us fix a uniformizer $z_i$ in $\niceO_i := \niceO_{X^{i-1} , x_i}$. We will attach, to any $f \in F^{\times}$, a sequence of integers $a_1 , \ldots , a_n$ as follows. Note that the residue field of $\niceO_{i-1}$ can be identified with the fraction field of $\niceO_i$. We write $f = z_1^{a_1} u_1$  where $u_1$ is a unit in $\niceO_1$. Considering the residue class of $u_1$ as an element of the fraction field of $\niceO_2$, we proceed to write $u_1 = z_2^{a_2} u_2$  where $u_2$ is a unit in $\niceO_2$. We continue in this way to construct the sequence $a_1 , \ldots , a_n$. Note that, generally speaking, this sequence depends on the choice of uniformizers $z_1 , \ldots , z_n$.

Let $f_1 , \ldots , f_{n+1} \in F^{\times}$. Write $a_{i1} , \ldots , a_{in}$ for the sequence of integers assigned to $f_i$ as above. Construct the $(n+1)\times n$ matrix $A = (a_{ij})$. Set $A_i$ to be the determinant of the $n\times n$ matrix that we get from $A$ by omitting the $i$-th row. Set $A^k_{ij}$ to be the determinant of the $(n-1) \times (n-1)$ matrix that we get from $A$ by deleting the $i$-th and $j$-th rows, and the $k$-th column. Set $B = \sum_k \sum_{i < j} a_{ik} a_{jk} A^k_{ij}$.

\begin {claim}

$$\mu^{n+1}_{\niceF} (f_1 , \ldots , f_{n+1}) = N_{k(x_n)/k} \left( (-1)^B ( \prod_{1\leq i \leq n+1} f_i^{(-1)^{i+1} A_i } ) (x_n) \right).$$

\end {claim}

By applying the abstract reciprocity law, we recover the Parshin reciprocity law (see \cite[Appendix A]{Sop}).

\paragraph {The Parshin higher residue reciprocity law} \label{par:Parshin residue law}

Considering $k_{\epsilon} := k[\epsilon_1 , \ldots , \epsilon_{n+1}] / (\epsilon_1^2 , \ldots , \epsilon_{n+1}^2)$ and $X_{\epsilon},B_{\epsilon}$, etc., similarly to \ref{par:residue law}, and considering the map $\mu^{n+1} : K_{n+1}(F_{\epsilon}) \to K_1(k_{\epsilon})$, one can derive, in principle, the higher residue reciprocity law (\cite[Appendix A]{Sop}), although we don't spell out the details in this text.


\section{Construction of $\mu_C$ and proof of the abstract reciprocity law} \label{Sec:ConstructionAndProof}

\subsection {Construction of $\mu_C$}
\label{SubSec:Construction}

Let us recall the codimension filtration in $K$-theory (\cite[(10.3.6)]{TT}). Write $S^d K(X)$ for the homotopy colimit of the spectra $K(X \ on \ Z)$, where $Z$ runs over closed subsets of $X$ of codimension $ \ge d$. Also, write

$$ Q^d K(X) := \bigvee_{p \in |X|^d} K(X_p \ on \ p). $$

Then we have the following homotopy fiber sequence:

$$ S^{d+1} K(X) \xrightarrow{} S^d K(X) \xrightarrow{p_d}  Q^d K(X) \xrightarrow{\partial_d} \Sigma S^{d+1} K(X).$$

Let us define $\Psi^d$ to be the composition:

$$ \Psi^d :  Q^d K(X) \xrightarrow{\partial_d} \Sigma S^{d+1} K(X) \xrightarrow{p_{d+1}} \Sigma  Q^{d+1} K(X). $$

Also, given a collection $C = (C^i)_{0 \leq i \leq n}$ ($C^i \subset |X^i|$), let us define a map

$$ sel_{C^{d}} :  Q^d K(X) \to  Q^d K(X), $$

given by projecting on summands corresponding to $p \in C^d$.

Let us define a map

$$ I : Q^n K(X) \to K(B);$$

In order to do this, we need to define maps $K(X_p \ on \ p) \to K(B)$, which we do by pushing-forward along $X_p \to B$. Technically, as a model for $K(X_p \ on \ p)$ we take perfect complexes on $X_p$ which are acyclic outside of the closed point $p$ (\cite[Section 3]{TT}), and as a model for $K(B)$ we take cohomologically bounded complexes of quasi-coherent sheaves with coherent cohomology on $B$ (\cite[Section 3]{TT}). Note that the important thing here is that coherent sheaves on $X_p$ which are supported on $p$ remain coherent after pushing-forward to $B$ (by our assumption that for $p \in |X|^n$, the map $Spec(X_p) \to Spec(B)$ is finite).

Finally, we define $\mu_C$ as follows\footnote{we assume that $C^0 = \{ \gamma \}$.} \footnote{In this formula, as we compose, the target becomes more and more suspended; we do not write the obvious suspensions, by abuse of notation.}:
\nopagebreak
$$ \mu_C = I \circ sel_{C^n} \circ \Psi^{n-1} \circ \ldots \circ \Psi^1 \circ sel_{C^1} \circ \Psi^0 .$$

\subsection {Proof of the reciprocity law} \label{SubSec:ProofRecLaw}

Let us show part (1) of theorem \ref{MainTheorem}. First, consider $d \neq n$. Then, we notice that $$sel_{C(\niceG)^d} \circ \Psi^{d-1} \circ sel_{C(\niceG)^{d-1}} = \Psi^{d-1} \circ sel_{C(\niceG)^{d-1}},$$ so that our formula for $\mu_{C(\niceG)}$ contains $$\Psi^{d} \circ sel_{C(\niceG)^d} \circ \Psi^{d-1} \circ sel_{C(\niceG)^{d-1}} = \Psi^d \circ \Psi^{d-1} \circ sel_{C(\niceG)^{d-1}}.$$ But $\Psi^d \circ \Psi^{d-1} = 0$, since we are dealing with a composition of two consequent arrows in a long exact sequence.

Let us show the case $d = n$ of part (1). Write $Y = \overline{x_{n-1}}$. We will deal first with the case $X = Y$, to simplify matters.

Note that $\mu_{C(\niceG)}$ equals the composition on the top horizontal line of the following commutative diagram:

$$ \xymatrix{
Q^0 K(X) \ar[r]^{\partial_0} & \Sigma S^1 K(X) \ar[r]^{p_1} \ar[d]^{i} & \Sigma Q^1 K(X) \ar[r]^{I} & \Sigma K(B) \\
& \Sigma S^0 K(X) \ar[rru]_{\tilde{I}} & &
}$$

Here, $i$ is the natural arrow, and $\tilde{I}$ is the arrow induced by pushforward. The crucial moment here is that $X$ is proper; Thus pushing-forward preserves coherence, which in turn enables us to construct the map $\tilde{I}$ on $K$-spectra. Now, noticing that $i \circ \partial_0 = 0$ (as a composition of two consequent arrows in a long exact sequence), finishes the proof.

In general (not assuming $X = Y$), we want to do the same as in the case $X=Y$, but working with $(X \ on \ Y)$-versions. To proceed, one considers the commutative diagram:

$$ \xymatrix{
Q^{n-1} K(X) \ar[rd]\ar[r]^{sel_{C(\niceG)^{n-1}}} & Q^{n-1} K(X) \ar[r]^{\partial_0} & \Sigma Q^n K(X) \ar[r]^{I} & \Sigma K(B) \\
& Q^{n-1} K(X \ on \ Y) \ar[u]\ar[r]^{\partial^{Y}_{n-1}} & \Sigma Q^{n} K(X \ on \ Y) \ar[u]\ar[ru]^{I^Y} &
}$$

and shows $I^Y \circ \partial^{Y}_{n-1} = 0$ as before.

Let us now show part (2) of theorem \ref{MainTheorem}. We note that the map $sel_{C(\niceG)^d}$ is the sum of the maps $sel_{C(\niceF)^d}$ (where $\niceF > \niceG$). Thus, the desired follows using claims \ref{claim_sum_composition} and \ref{claim_sum_projections}.


\section {Calculation of local symbols} \label{Sec:Calc}

In this section, we calculate some symbol maps for local schemes. Using lemma \ref{LemmaLocalIsEnough}, these calculations imply the claims of subsection \ref{SubSect:ConcreteRecLaw}.

Let us fix the following notations and assumptions for this section. Let $k$ be a field, and let $B = Spec(k)$. Also, let $A$ be a regular Noetherian local $k$-algebra, and set $X=Spec(A)$. Denote by $\liem$ the maximal ideal of $A$, and $k' = A / \liem$. We assume that $k'$ is finite over $k$. We denote by $F$ the fraction field of $A$.

\subsection{The case $dim(X)=1$}

In this subsection, we additionally assume that $A$ is of Krull dimension $1$. Let $v: F^{\times} \to \Z$ be the valuation, and let $ [ \cdot ] : A \to k'$ be the quotient map. Finally, choose a uniformizer $z \in A$ (i.e., $v(z)=1$).

Consider the unique full flag $\niceF: \ p_{\liem} < p_{0}$ in $X$. We have the corresponding symbol map

$$ \mu = \mu_{\niceF} : K(F) \to \Sigma K(k). $$

We write $\mu^i$ for the induced map $K_i (F) \to K_{i-1} (k)$.

\subsubsection {The degree}

\begin {claim}

The morphism $F^{\times} \cong K_1(F) \xrightarrow{\mu^1} K_0(k) \cong \Z$ is equal to $[k':k] \cdot v$.

\end {claim}

\begin {proof}

Since the composition $K_1(A) \to K_1(F) \to K_0(A \ on \ \liem)$ is zero (as part of a long exact sequence), it is enough to prove that $$F^{\times} \cong K_1(F) \to K_0(A \ on \ \liem) \to K_0(k) \cong \Z$$ maps $z$ to $[k':k]$. By lemma \ref{KthLem1}, the image of $z$ under the above map is equal to the  alternating sum of dimensions (over $k$) of cohomologies of the complex $$ \xymatrixrowsep{0.3pc}\xymatrix{A \ar[r]^{z} & A \\ -1 & 0} $$ which is $[k':k]$.

\end {proof}

\subsubsection {The tame symbol}

\begin {claim} \label{ClaimLocalTameSymbolCalc}

The morphism $$F^{\times} \wedge F^{\times} \cong K_1(F) \wedge K_1(F) \to K_2(F) \xrightarrow{\mu^2} K_1(k) \cong k^{\times}$$ is given by:

$$ f \wedge g \mapsto N_{k' / k} \left( (-1)^{v(f) \cdot v(g)} \left[ \frac{f^{v(g)}}{g^{v(f)}} \right] \right). $$

\end {claim}

\begin {proof}

We call the above morphism $F^{\times} \wedge F^{\times} \to k^{\times}$, by abuse of notation, $\mu^2$. By bilinearity and anti-symmetry of $\mu^2$, it is enough to verify:

\begin {enumerate}[(i)]

\item For $f,g \in A^{\times}$, $\mu^2 (f \wedge g) = 0$.

\item For $f \in A^{\times}$, $\mu^2 (f \wedge z) = N_{k'/k}( [ f ] )$.

\item $\mu^2 (z \wedge z) = N_{k'/k} (-1)$.

\end {enumerate}

As for the first item, it follows since the following composition is zero (being a part of the localization long exact sequence):

$$ K_2 (A) \to K_2 (F) \to K_1 (A \ on \ k').$$

As for the second item, consider the commutative diagram:

$$ \xymatrix { K_1(A) \wedge K_1 (F) \ar[d] \ar[rr] & & K_1(A) \wedge K_0 (A \ on \ k') \ar[d] \\ K_1(F) \wedge K_1(F) \ar[r] & K_2(F) \ar[r] & K_1(A \ on \ k') \ar[r] & K_1(k) }$$

We have the element $f \wedge z$ in the upper-left group $ K_1(A) \wedge K_1 (F) $, and we should walk it through down, and then all the way right. Using commutativity of the diagram, we can chase instead by the upper path, and using lemma \ref{KthLem2}, the result is represented by the automorphism of the following complex:

$$ \xymatrixrowsep{0.3pc}\xymatrix{A \ar[dd]_{f}\ar[r]^{z} & A \ar[dd]_{f} \\ & \\ A \ar[r]^{z} & A \\ -1 & 0} $$

Taking alternating determinant of cohomology, we see that the above automorphism represents the element $N_{k'/k} ([ f ]) \in k^{\times} \cong K_1(k)$. 

Let us handle the third item on our list. Denote by $\{ \cdot , \cdot \} : K_1(F) \wedge K_1(F) \to K_2(F)$ the multiplication in $K$-theory. Recall the Steinberg relation: $$\{ x, 1-x\} = 0$$  for $x , 1-x \in F^{\times} \cong K_1(F)$. We then calculate: $$ \{z,z\}=\{z,(1-z^{-1})^{-1}\}\{z,1-z\}\{z,-1\}= \{z^{-1},1-z^{-1}\}\{z,1-z\}\{z,-1\} = \{z,-1\} $$ (this calculation appears in \cite[Theorem 2.6]{Sn}). Hence, $ \mu^2 (z \wedge z) = \mu^2 (-1 \wedge z) = N_{k'/k} (-1)$, by (ii) above.
\end {proof}

\subsubsection {The residue}

In this subsection, $k$ is assumed to be infinite\footnote{We suspect that the assumption of $k$ being infinite is unnecessary; Nevertheless, we use this assumption in the proof.}. We consider a base change of our setup from $k$ to $k_{\epsilon} := k [ \epsilon_1 , \epsilon_2 ] / (\epsilon_1^2 , \epsilon_2^2 )$. Thus, we have $A_{\epsilon} := k_{\epsilon} \otimes_k A$, and similarly $F_{\epsilon}$, $X_{\epsilon}$, $B_{\epsilon}$, etc. Hence, the basic morphism of schemes from which we build the symbol map is now $X_{\epsilon} \to B_{\epsilon}$.

\begin {claim}

The morphism $$F_{\epsilon}^{\times} \wedge F_{\epsilon}^{\times} \cong K_1(F_{\epsilon}) \wedge K_1(F_{\epsilon}) \to K_2(F_{\epsilon}) \xrightarrow{\mu_{\epsilon}^2} K_1(k_{\epsilon}) \cong k_{\epsilon}^{\times}$$ sends $ (1-\epsilon_1 f) \wedge (1-\epsilon_2 g)$ to $1-\epsilon_1 \epsilon_2 Res(fdg)$ (where $f,g \in F$).

\end {claim}

\begin {proof} In this proof let us denote by $\mu^2$ the morphism $F_{\epsilon}^{\times} \wedge F_{\epsilon}^{\times} \to k_{\epsilon}^{\times}$ as in the statement.

\textit{(a)} Let us first show that the expression $\mu^2 (1-\epsilon_1 f , 1-\epsilon_2 g)$ is of the form  $1-\epsilon_1 \epsilon_2 R(f,g)$, where $R(f,g) \in k$. In other words, the "constant term" is $1$, and there are no "linear terms". Towards this end, we perform "base change", sending $\epsilon_2 \mapsto 0$. The operation $\mu^2$ commutes with such base change. We depict it as follows:

$$ \xymatrix { \mu^2 (1-\epsilon_1 f , 1-\epsilon_2 g) \ar@{|->}[d] \ar@{=}[r] & a+b\epsilon_1 +c\epsilon_2+d\epsilon_1 \epsilon_2 \ar@{|->}[d] \\ \mu^2 (1-\epsilon_1 f , 1 ) \ar@{=}[r] & a+b\epsilon_1}$$

Here, the vertical assignment is base change, from $k_{\epsilon}$ to $k_{\epsilon}/(\epsilon_2)$. Note that the lower-left element is $1$ (by bi-multiplicativity of $\mu^2$), so that we get $a=1$ and $b=0$. Similarly one gets $c=0$.

\textit{(b)} We notice that $R(f,g)$ is bi-linear. The bi-additivity follows immediately from the bi-multiplicativity of $\mu^2$ and \textit{(a)}. Next, let us show that $R(\alpha f,g) = \alpha R(f,g)$ for every $\alpha \in k$ (the homogeneity in the second variable is shown analogously). In case $\alpha = 0$, it is clear. Otherwise, we get the equality by performing "base change", sending $\epsilon_1 \mapsto \alpha^{-1} \epsilon_1$.

\textit{(c)} We now show the following:

\begin {enumerate}[(1)]
\item $R(z^n,z^m) = 0$ for $n,m \in \Z$, provided $n+m \neq 0$.
\item $R(z^{-n},z^n) = n$ for $n \in \Z$.
\item $R(z^{-n} , f) = 0$ for $n \in \Z_{\ge 0}$, provided that $v(f) \gg n$.
\end {enumerate}

To show the first item, let us first notice that by using lemma \ref{LemmaPassToCompletion}, it is enough to verify the claim in the case when $A$ is complete, and so we do. Consider the automorphism $z \mapsto \alpha z$, where $\alpha \in k^{\times}$. We notice that it does not alter the symbol $\mu^2$, since it commutes with passing to the quotient $A \mapsto A / \liem$. Thus, we have $R(z^n,z^m)=R((\alpha z)^n,(\alpha z)^m)$. By bilinearity (\textit{(b)}), we get $R(z^n,z^m) =\alpha^{n+m} R(z^n,z^m)$. Choosing $\alpha$ so that $\alpha^{n+m} \neq 1$, we conclude $R(z^n,z^m)=0$. Such a choice of $\alpha$ is possible since $k$ is infinite and $n+m \neq 0$.

To show the second item, note that:

$$ \mu^2 (1-\epsilon_1 z^{-n} , 1 - \epsilon_2 z^n ) = \frac{\mu^2 (z^n-\epsilon_1 , 1 - \epsilon_2 z^n )}{\mu^2 (z^n , 1 - \epsilon_2 z^n )},$$

and hence it is enough to calculate $\mu^2 (z^n-\alpha \epsilon_1 , 1 - \epsilon_2 z^n )$ (where $\alpha \in k$). By lemmas \ref{KthLem1} and \ref{KthLem2}, we should calculate the determinant of multiplication by $1 - \epsilon_2 z^n$ on the cohomology of:

$$ \xymatrixrowsep{0.3pc}\xymatrix{A_{\epsilon} \ar[r]^{z^n-\alpha \epsilon_1} & A_{\epsilon} \\ -1 & 0}.$$

The only non-zero cohomology is the $0$-th one. It is a free $k_{\epsilon}$-module (with basis $1,z,\ldots, z^{n-1}$). Multiplication by $1 - \epsilon_2 z^n$ is just multiplication by $1-\alpha \epsilon_1 \epsilon_2$. Thus, the determinant equals $(1-\alpha \epsilon_1 \epsilon_2)^n = 1-n\alpha \epsilon_1 \epsilon_2$, and consequently $R(z^{-n},z^n)=n$.

The third item is verified similarly to the second one (when $v(f)\gg n$, the operator whose determinant we should consider is just identity).

\textit{(d)} By breaking $f$ and $g$ into sums of monomials in $z$ and a reminder of large enough valuation, the proposition follows from \textit{(a)},\textit{(b)}, and \textit{(c)} .

\end {proof}

\begin{remark}
One could obtain the residue symbol differently, by considering $k_{\epsilon} := k[\epsilon] / (\epsilon^3)$. Then, $\mu^2 (1-\epsilon f , 1- \epsilon g) = 1 - \epsilon^2 Res(fdg)$.
\end{remark}

\subsection{The case $dim(X)>1$}

In this subsection, we drop the assumption that $A$ is $1$-dimensional; We denote the Krull dimension of $A$ by $n$. 

\subsubsection {The Parshin symbol}

Fix a full flag $$\niceF : x_n < \ldots < x_0$$ in $X$, corresponding to a chain of prime ideals $$0 = \liep_0 \subsetneq \ldots \subsetneq \liep_n = \liem.$$ Consider $X^i := \overline{x_i}$ as an integral closed subscheme of $X$. We obtain a symbol:

$$ \mu = \mu_{\niceF} : K(F) \to \Sigma^n K(k). $$

As in \ref{par:Parshin law} consider the resulting map $\mu^{n+1}_{\niceF} : \bigwedge^{n+1} F^{\times} \to k^{\times}$. In \ref{par:Parshin law} we essentially wrote a formula for this map (which we now want to verify), under the assumption that our flag is regular. In order to compute this map "recursively", we will use Quillen's devissage (lemma \ref{LemmaDevissage}) -- application of which will be possible due to regularity of $\niceF$.

\begin {claim}

The symbol $\mu_{\niceF}:K(F_X)\to \Sigma^n K(k)$ equals to the following composition:

$$\xymatrix{ K(F_X)\ar[r] & \Sigma K(X_{x_1} \ on \ x_1)  &  \Sigma K(F_{X^1}) \ar[r] \ar[l]_{\ \ \ \ \ \sim} & \Sigma^2 K(X^1_{x_2} \ on \ x_2) \\
K(F_{X^2}) \ar[r] \ar`u[r]`[urrr]^{\sim}[urrr]& \ldots \ar[r] &  \Sigma^n K(F_{X^n})\ar[r] &  \Sigma^n K(k), } $$

where the arrows $\xleftarrow{\sim}$ stand for Quillen's devissage.

\end {claim}

In view of this claim, $\mu^{n+1}_{\niceF}$ equals to the composition

$$\bigwedge^{n+1} F_{X}^{\times} \to K_{n+1}(F_X) \xrightarrow{\partial_0} K_n(F_{X_1}) \xrightarrow{\partial_1} \ldots \xrightarrow{\partial_{n-1}} K_1(F_{X_n}) \to K_1(k)$$

where $\partial_i$ is the composition of the boundary map and the inverse of the devissage. 

The following lemma will allow us, in principle, to calculate $\mu^{n+1}_{\niceF} (f_1 , \ldots , f_{n+1})$ for any $f_1 , \ldots , f_{n+1} \in F^{\times}$.

\begin {lemma}

Let $R$ be a $1$-dim. regular local Noetherian ring with maximal ideal $\lien$, residue field $\ell$ and fraction field $L$. Let $z \in R$ a uniformizer. Consider the map $$K(L) \to \Sigma K(R \ on \ \lien ) \xleftarrow{\sim} \Sigma K(\ell)$$ -- composition of boundary with devissage. Using it we construct a map $$\nu^m : \bigwedge^m L^{\times} \to K_m (L) \to K_{m-1} (\ell).$$ The following hold:

\begin {enumerate}[(i)]
\item $\nu^m (f_1 , \ldots , f_m) = 0$ for $f_1 , \ldots , f_m \in R^{\times}$.
\item $\nu^m (f_1 , \ldots , f_{m-2} , z, z) = \nu^m (f_1 , \ldots , f_{m-2} , -1, z)$ for $f_1 , \ldots , f_{m-2} \in R^{\times}$.
\item $\nu^m (f_1 , \ldots , f_{m-1} , z) = [f_1] \wedge \ldots \wedge [f_{m-1}]$ for $f_1 , \ldots , f_{m-1} \in R^{\times}$ (we recall that $[ f ]$ denotes the residue in $\ell^{\times}$ of $f\in R^{\times}$, considered as an element of $K_1(\ell)$ in the case at hand).
\end {enumerate}

\end {lemma}

\begin {proof}

The first item is clear, since $\nu^m (f_1 , \ldots , f_m)$ is the value of the composition $K_{m} (R) \to K_{m} (L) \to K_{m-1} ( R \ on \ \lien)$ on $ f_1 \wedge \ldots \wedge f_m  \in K_{m} (R)$, and the composition is zero as part of a long exact sequence.

The second item follows from Steinberg relation (like in the proof of claim \ref{ClaimLocalTameSymbolCalc}).

The third item follows from the commutativity of the following diagram:

$$ \xymatrix { K_{m-1}(R) \wedge K_1 (L) \ar[r] \ar[dd] & K_{m-1} (R) \wedge K_0 (R \ on \ \lien) \ar[dd] & K_{m-1}(R) \wedge K_0 (\ell) \ar[l]_{\ \ \ \sim} \ar[d] \\ & & K_{m-1}(\ell) \wedge K_0 (\ell) \ar[d] \\ K_m(L) \ar[r] & K_{m-1} (R \ on \ \lien) & K_{m-1}(\ell) \ar[l]_{\sim} }$$

Here the left square commutes as the boundary morphism is a morphism of $K(A)$-modules, while the right square commutes as Quillen's devissage morphism is a morphism of $K(A)$-modules.

Note that the element $\nu^m (f_1 , \ldots , f_{m-1} , z)$ is the result of going right on the lower line, applied to $f_1 \wedge \ldots \wedge f_{m-1} \wedge z$. However, this element comes from an element at the upper-left corner, which we can chase through the right on the upper line, and then to the lower-left corner through the right line.

\end {proof}

\subsection{Auxiliary lemmas} \label{SubSectAuxLemm}

We state two lemmas which are used above, and whose proof is straightforward.

\begin {lemma} \label{LemmaLocalIsEnough} Let $X \to B$ be as in \ref{AssumptionsOnXandB}. Let $$\niceF: x_n < x_{n-1} < \ldots < x_0 = \gamma$$ be a full flag of points in $|X|$. Writing $p := x_n$, we consider also the setting $X_p \to B$ and the obvious flag $\niceF_p$ on $X_p$ induced by $\niceF$. We have two symbol maps; $$\mu_{\niceF}: K(F) \to \Sigma^n K(k)$$ and $$ \mu_{\niceF_p}: K(F) \to \Sigma^n K(k) $$ (note that the function field of $X_p$ is identified with $F$). Then these two symbol maps are equal.
\end {lemma}

\begin {lemma} \label{LemmaPassToCompletion}
Let $A$ be a $1$-dim. regular local Noetherian $k$-algebra whose residue field is finite over $k$, and let $\hat{A}$ be its completion. Write, as usual, $X=Spec(A)$ and $B=Spec(k)$, and write $\hat{X}=Spec(\hat{A})$. Also, denote by $F$ (resp. $\hat{F}$) the fraction field of $A$ (resp. $\hat{A}$). Associated to the unique full flag in $X$ (resp. $\hat{X}$) we have the symbol $K(F) \to \Sigma K(k)$ (resp. $K(\hat{F}) \to \Sigma K(k)$). Then the diagram

$$ \xymatrix{K(F) \ar[rr] \ar[rd] & & K(\hat{F}) \ar[ld] \\ & \Sigma K(k) &} $$

commutes.

\end {lemma}


\appendix
\section{Infinite sums of maps of spectra} \label{App:AppendixA}

\newcommand{\spectra}{\mathcal{S}p}

In this text, we consider spectra as a triangulated category $\spectra$. We recall that a spectrum is called compact, if maps from it commute with small direct sums. An example of a compact spectrum is $\Sigma^k \bbS$, a suspension of the sphere spectrum. The following definitions are inspired by \cite[Appendix A]{Cl}.

\begin {definition}

Let $f_i: \niceS \to \niceT$ ($i \in I$) be a family of maps of spectra, and $f: \niceS \to \niceT$ an additional map. We say that \textbf{$f$ is the sum of the $f_i$} (written $f = \sum_{i\in I} f_i$) if for every compact spectrum $\niceC$, and every element $e \in Hom_{\spectra}( \niceC , \niceS )$, almost all (i.e., all but finitely many) of the maps $f_i \circ e$ are equal to zero, and the sum of all these $f_i \circ e$ is equal to $f \circ e$.

\end {definition}

We note that we don't claim uniqueness of the sum (in whatever sense). In reality, this notion of "summability and summation on compact probes" is derived from a more holistic notion:

\begin {definition}

Let $f_i: \niceS \to \niceT$ ($i \in I$) be a family of maps of spectra, and $f: \niceS \to \niceT$ an additional map. An \textbf{evidence} for $f$ being the sum of the $f_i$ is a map:

$$ g: \niceS \to \bigvee_{i \in I} \niceT, $$

so that when we compose $g$ with the projection to the $i$-th summand we get $f_i$, while when we compose $g$ with the fold map, we get $f$.

\end {definition}

The following is evident:

\begin {claim}

Existence of an evidence for $f$ being the sum of the $f_i$ implies that $f$ is the sum of the $f_i$.

\end {claim}

Let us also note the following two auxiliary claims (whose proof is straightforward):

\begin {claim} \label{claim_sum_composition}

Let $h : \niceU \to \niceS , g : \niceT \to \niceV$. If $f$ is the sum of the $f_i$ (we have evidence for $f$ being the sum of the $f_i$) , then $g \circ f \circ h$ is the sum of the $g \circ f_i \circ h$ (we have evidence for $g \circ f \circ h$ being the sum of the $g \circ f_i \circ h$).

\end {claim}

\begin {claim} \label{claim_sum_projections}

Let $\niceS_i$ ($i \in I$) be a collection of spectra, and write $\niceS = \bigvee_{i \in I} \niceS_i$. Then we have evidence for $id$ being the sum of $pr_i$ ($i \in I$), where $id$ is the identity morphism of $\niceS$, while $pr_i$ is the morphism of projection on the $i$-th summand. In particular, $id = \sum_{i \in I} pr_i$.

\end {claim}

\section{$K$-theory calculation lemmas} \label{App:AppendixB}

We state some lemmas, which will be of use when calculating the concrete symbols. In what follows, $X$ is a Noetherian scheme, $U \subset X$ an open subscheme, and $Z$ the closed complement.

We denote by $SPerf(X)$ the category of (strictly) bounded complexes of $\niceO_X$-modules, whose terms are locally free of finite rank. By $SPerf(X \ on \ Z)$ we denote the full subcategory of $SPerf(X)$ consisting of complexes whose cohomologies are supported on $Z$.

\begin {fact}

There is a natural map from (the geometric realization of) the core groupoid of $SPerf(X)$ to $K(X)$. In particular, every object in $SPerf(X)$ defines a point in $K(X)$. In addition, the automorphism group of any object of $SPerf(X)$ maps into $K_1(X)$; Especially, since $\niceO(X)^{\times}$ maps into the automorphism group of the object $\niceO_X \in SPerf(X)$, one has a map $\niceO(X)^{\times} \to K_1(X)$. Thus, given an object or an automorphism in $SPerf(X)$, one can view it as an element of an appropriate $K$-group $K_i(X)$. We will abuse this without further notice.

\end {fact}

\begin {claim}

Let $X$ be local (i.e., the spectrum of a local ring). Then the above map $\niceO(X)^{\times} \to K_1(X)$ is an isomorphism.

\end {claim}

\begin {lemma} \label{KthLem1}

Let $f \in \niceO(X)$ be such that $f|_{U}$ is invertible. Then the image of $f|_{U} \in \niceO(U)^{\times}$ under the map $K_1(U) \to K_0 (X \ on \ Z)$ which is obtained from the localization sequence (\cite[Theorem 7.4]{TT})

$$ K(X \ on \ Z) \to K(X) \to K(U) $$

is given by the complex

$$ \xymatrixrowsep{0.3pc}\xymatrix{\niceO_X \ar[r]^{f} & \niceO_X .\\ -1 & 0} $$

\end {lemma}

\begin {lemma} \label{KthLem2}

Let $f \in \niceO(X)^{\times}$, and $C \in SPerf (X \ on \ Z)$. Then the image of $f \wedge C$ under the product map $K_1(X) \wedge K_0 (X \ on \ Z) \to K_1 (X \ on \ Z)$ is given by the automorphism $C \otimes \niceO_X \xrightarrow{1 \otimes f} C \otimes \niceO_X$.

\end {lemma}

\begin {lemma}[Quillen's devissage] \label{LemmaDevissage}
Suppose that $X$ and $Z$ are regular. Then the morphism $K(Z) \to K(X \ on \ Z)$ (induced by pushforward) is an equivalence of spectra.
\end {lemma}


\pagebreak

\end{document}